\documentclass[11pt,a4paper]{article}
\usepackage{amsmath,amsthm,amsfonts,amssymb}
\usepackage{array,fullpage, epic, eepic}
\usepackage{graphicx}
\usepackage{color}
\usepackage{float}
\usepackage{algorithm}
\usepackage{algorithmic}

\theoremstyle{plain}

\newtheorem{lem}{Lemma}[section]

\newtheorem{theo}[lem]{Theorem}

\newtheorem{cor}[lem]{Corollary}

\theoremstyle{definition}

\def\DJ{{\hbox{D\kern-.8em\raise.15ex\hbox{--}\kern.35em}}}

\title{Wythoff-Fibonacci Sequences and a Perturbed Greedy Almost-involution 
}
\author{Luis Mart\'\i nez$^{1}$, Iker Malaina.
\\
University of the Basque Country (EHU), Department of Mathematics,\\ 48080 Bilbao, Spain \\
E-mail: \ luis.martinez@ehu.eus,\ iker.malaina@ehu.eus \\
$^1$ Corresponding author
}
\date{}

\begin{document}

\maketitle

\begin{abstract} 
We introduce the lower and upper Wythoff-Fibonacci sequences, obtained from the classical Wythoff sequences by a Fibonacci correction. Specifically, if we put $$\epsilon(j)=\begin{cases}(-1)^k, & \text{if }j=F_k\text{ for some }k\\ 0, & \text{in other case}\end{cases},$$ where $F_k$ is the $k$-th Fibonacci number, then we define the general terms of the lower and upper Wythoff-Fibonacci sequences by $$LWF(n)=\begin{cases} 1, & \text{if }n=1,\\
3, & \text{if }n=2,\\
a(n)+\epsilon(n), & \text{if }n\geq 3.\end{cases}$$ and $$UWF(n)=\begin{cases} 2, & \text{if }n=1,\\
b(n)+\epsilon(n), & \text{if }n\geq 2,\end{cases}$$ respectively. We show that these sequences partition the set of natural numbers and use them to give an explicit formula for a sequence $q^{\star}_j$, defined from a greedy construction studied by the first author and his coauthors in a previous paper, but with the additional condition that $q^{\star}_1=3$, instead of being defined by the greedy rule. This sequence is a permutation of the set of non-negative integers and has the property that every integer appears exactly once in the sequence of differences $q^{\star}_j-j$. We prove that $q^{\star}_{q^{\star}_j}=j\ \forall j\geq 5$, so that $q^{\star}_j$ is an almost-involution. We also give another greedy algorithm generating $q^{\star}_j$.
\end{abstract}



\section{Introduction}\label{intr}

Combinatorics has many applications in both pure and applied mathematics; in pure mathematics, for example, in group theory (\cite{s}), and in applied mathematics, for instance, in the computational design of vaccines (\cite{mmlmmf}).

Fibonacci numbers have also applications in combinatorics, including design theory (\cite{ward}), the theory of latin squares (\cite{bch}), or coding theory (\cite{bmm}).

In relation with the interplay between Fibonacci numbers and combinatorics, the first author, jointly with other coauthors, introduced in \cite{bmmmv} a greedy algorithm which produce an infinite array, in order to search infinite orthogonal arrays with a regular automorphisms group isomorphic to the additive group of $\mathbb Z$.

The algorithm generating the array is:

\begin{algorithm}[H]
\caption{}\label{alg}
\begin{algorithmic}[1]
\REQUIRE $n,m\in\mathbb N$
\ENSURE $Q^{[m,n]}\in M_{(m+1)\times n}(\mathbb Z^{\geq 0})$
\FOR{$i\in\{0,\dots,m\}$}
\FOR{$j\in\{0,\dots,n-1\}$}
\STATE $S=\{q^{[m,n]}_{i,k}+q^{[m,n]}_{l,j}-q^{[m,n]}_{l,k}\mid 0\leq l<i,0\leq k<j\}$
\STATE $q^{[m,n]}_{i,j}=\text{mex}(S\cap\mathbb Z^{\geq 0})$
\ENDFOR
\ENDFOR
\RETURN $Q^{[m,n]}$
\end{algorithmic}
\end{algorithm}

The matrices $Q^{[m,n]}$ defined for given $m$ and $n$ are monotone with respect to the inclusion, and therefore let construct an array with an infinite number of rows and columns.

In particular, the authors gave explicitely the values of the third row, with general term \begin{equation}\label{gvqj}q_{j}=\begin{cases} 0, & \text{if }j=0\\
\lfloor \varphi j\rfloor +1, & \text{if }j\in A\\ \lfloor (\varphi -1)j\rfloor, & \text{ if }j\in B\end{cases},\end{equation} where $A$ and $B$ are the ranges of the lower and upper Wythoff sequences, respectively. They observed also that the sequence $q_j$, whose first terms are $0, 2, 1, 5, 7, 3, 10, 4, 13, 15$, corresponds with sequence A002251 in Sloane's on-line encyclopedia of integer sequences \cite{sol}, obtained by swapping $a(k)$ and $b(k)$ for all $k\geq 1$, and that when $q_j+1$ is taken we obtain the sequence studied on page 76 in \cite{berconguy}, which is also studied by  A. Shapovalov (\cite{shap}) and by Venkatachala (\cite{venk}) (sequence A019444 in Sloane's database).

In this paper we introduce a perturbation in the greedy sequence that we have described, producing a sequence $q^{\star}_j$ in which we force that $q^{\star}_1=3$, overimposing the greedy rule for this value. We give an explicit formula for the values of $q^{\star}_j$. Unlike in (\ref {gvqj}), where the casuistic depends on the partition induced by the classical lower and upper Wythoff sequences,we introduced for the study of this new sequence a variation of these Wythoff sequences, that we call lower and upper Wythoff-Fibonacci sequences, based in the classical ones having into account a Fibonacci correction.

In contrast to the case of $q_j$, which is an involution, the sequence $q^{\star}_j$ satisfies $q^{\star}_{q^{\star}_j}=j$ only when $j=0$ or $j\geq 5$, and we call it an almost-involution.

\section{Wythoff-Fibonacci sequences}

Let us stablish some notation:\newline

$\mathbb N=\{1,2,3,\dots\}$ will denote the set of natural numbers and $\mathbb Z^{\geq 0}=\{0,1,2,\dots\}$ the set of non-negative integers.

More generally, if $n\in\mathbb N$, $\mathbb Z^{\geq n}=\{n,n+1,n+2,\dots\}$ will denote the set of natural numbers greater than or equal to $n$.

$\varphi=\frac{1+\sqrt{5}}{2}$ will denote the golden ratio.

For $n\in\mathbb N$, $a(n)=\lfloor\varphi n\rfloor$ and $b(n)=\lfloor\varphi^2 n\rfloor$ will denote the general terms of the lower Wythoff sequence and the upper Wythoff sequence, respectively.

For $n\in\mathbb N$, $F_n$ will denote the $n$-th Fibonacci number.

For $m\in\mathbb N$, let $\mathfrak F_m=\{F_k\mid k\in\mathbb Z^{\geq m}\}$. In particular, we will denote by $\mathfrak F$ to $\mathfrak F_1$, that is, to the set of Fibonacci numbers.

We define also, for $j\geq 2$, $$\epsilon(j)=\begin{cases}(-1)^k, & \text{if }j=F_k\text{ for some }k\\ 0, & \text{in other case}\end{cases}.$$ Observe that, since $j\geq 2$, the $k$ in the definition is unique.

If $A\subseteq\mathbb Z^{\geq 0}$ we will denote, as usual in the literature, by $\text{mex}(A)$ (minimum excluded value) the number $\min (\mathbb Z^{\geq 0}-A)$, that is, the minimum non-negative integer that is not in $A$.

Now we will introduce two new sequences related both to Wythoff sequences and the Fibonacci sequence, which we will call the lower and upper Wythoff-Fibonacci sequences LWF and UWF, respectively, defined for $n\in\mathbb N$ by $$LWF(n)=\begin{cases} 1, & \text{if }n=1,\\
3, & \text{if }n=2,\\
a(n)+\epsilon(n), & \text{if }n\geq 3.\end{cases}$$ and $$UWF(n)=\begin{cases} 2, & \text{if }n=1,\\
b(n)+\epsilon(n), & \text{if }n\geq 2.\end{cases}$$

We will denote by $A^{\star}$ and $B^{\star}$ the ranges of $LWF$ and $UWF$, that is, $$A^{\star}=\{1, 3, 5, 6, 7, 9, 11, 13, 14, 16,\dots\}$$ and $$B^{\star}=\{2, 4, 8, 10, 12, 15, 18, 21, 23, 26,\dots\}.$$

\begin{theo}\label{asbsp} The sets $A^{\star}$ and $B^{\star}$ form a partition of the set of natural numbers.
\end{theo}

We will use a lemma:

\begin{lem} If $k\in\mathbb Z^{\geq 2}$, then $$a(F_k)+(-1)^k=b(F_{k-1})$$ and $$b(F_k)+(-1)^k=a(F_{k+1}).$$
\end{lem}

\begin{proof} If we put $\varphi^{\star}=\frac{1-\sqrt{5}}{2}=-\varphi^{-1}$ then, by Binet's formula, $$F_k=\frac{\varphi^k-(\varphi^{\star})^k}{\sqrt{5}}\ \forall k\in\mathbb N.$$ Now, $$F_{k+1}-\varphi F_k=\frac{\varphi (\varphi^{\star})^k-(\varphi^{\star})^{k+1}}{\sqrt{5}}=\frac{(\varphi^{\star})^k(\varphi-\varphi^{\star})}{\sqrt{5}}.$$ Since $\varphi-\varphi^{\star}=\sqrt{5}$, this is $(\varphi^{\star})^k$. Therefore, \begin{equation}\label{pfk}\varphi F_k=F_{k+1}+(-1)^{k+1}\varphi^{-k}.\end{equation} Taking floor on both sides of (\ref{pfk}), we get \begin{equation}\label{afk} a(F_k)=F_{k+1}+\lfloor (-1)^{k+1}\varphi^{-k}\rfloor.\end{equation} Since $0<\varphi^{-k}<1$, we conclude that \begin{equation}\label{fifk} a(F_k)=\begin{cases} F_{k+1}-1, & \text{ if }k\text{ is even,} \\ F_{k+1}, & \text{if }k\text{ is odd.}\end{cases}\end{equation} Now, since $\varphi^2=\varphi+1$, $\varphi^2 F_k=\varphi F_k+F_k$. Taking floor on both sides, we obtain $b(F_k)=a(F_k)+F_k$ and, using (\ref{fifk}), \begin{equation}\label{sfifk}b(F_k)=\begin{cases} F_{k+1}+F_k-1=F_{k+2}-1, & \text{ if }k\text{ is even,} \\ F_{k+1}+F_k=F_{k+2}, & \text{if }k\text{ is odd.}\end{cases}\end{equation} The result follows now from (\ref{fifk}) and (\ref{sfifk}).
\end{proof}

Now we will prove the theorem:

\begin{proof} Since $LWF(1)=1, LWF(2)=3, UWF(1)=2, UWF(2)=4$, we have to prove that $\{ LWF(n)\mid n\geq 3\}$ and $\{ UWF(n)\mid n\geq 3\}$ partition the set $\mathbb Z^{\geq 5}=\{5,6,\dots\}$.

By the lemma, we have that $\{LWF(n)\mid n\geq 3\}$ is the union of the sets $$S_1=\{b(m)\mid m\in\mathfrak F_3\}=\{b(2), b(3), b(5),...\}$$ and $$S_2=\{a(m)\mid m\in\mathbb Z^{\geq 4}-\mathfrak F_5\}=\{a(4), a(6), a(7), a(9),...\}.$$ Similarly, $\{UWF(n)\mid n\geq 3\}$ is the union of the sets $$S_3=\{a(m)\mid m\in\mathfrak F_5\}=\{a(5), a(8), a(13),...\}$$ and $$S_4=\{b(m)\mid m\in\mathbb Z^{\geq 4}-\mathfrak F_5\}=\{b(4), b(6), b(7), b(9),...\}.$$ Clearly, the sets $S_1, S_2, S_3, S_4$ are mutually disjoint, and their union is $$\mathbb N-\{a(1), a(2), a(3), b(1)\}=\mathbb Z^{\geq 5}.$$
\end{proof}

\section{Greedy almost-involution}

The sequence $q_n$ mentioned in Section \ref{intr} obtained with the greedy Algorithm \ref{alg} takes in the $j$-th position the first non-negative number $q_j$ which is not already in the previous values of the sequence and that the difference $q_j-j$ is distinct from the previous differences $q_k-k$ with $k<j$. This forces that $q_1=2$.

In this section we will use the same greedy algorithm to produce a sequence $q^{\star}_j$ in which we disturb the value of $q^{\star}_1$ (the only non-greedy modification), that we take to be $3$, and taking again, for $j\geq 2$, the first non-negative number $q^{\star}_j$ which is not in the previous values of the sequence and that $q^{\star}_j-j$ is distinct from the previous differences $q^{\star}_k-k$ with $k<j$.

Specifically, $q^{\star}_0=0, q^{\star}_1=3$ and, if we define for $j\geq 2$ $$U^{\star}_j=\{q^{\star}_k\mid 0\leq k<j\},$$ $$W^{\star}_j=\{q^{\star}_k-k+j\mid0\leq k<j\}$$ and $$S^{\star}_j=U^{\star}_j\cup W^{\star}_j,$$ then $$q^{\star}_j=\text{mex}(S^{\star}_j\cap\mathbb Z^{\geq 0}).$$

The first $100$ values of the sequence are\newline

0, 3, 1, 4, 2, 8, 10, 12, 5, 15, 6, 18, 7, 21, 23, 9, 26, 28, 11, 31, 
33, 13, 36, 14, 39, 41, 16, 44, 17, 47, 49, 19, 52, 20, 55, 57, 22, 
60, 62, 24, 65, 25, 68, 70, 27, 73, 75, 29, 78, 30, 81, 83, 32, 86, 
88, 34, 91, 35, 94, 96, 37, 99, 38, 102, 104, 40, 107, 109, 42, 112, 
43, 115, 117, 45, 120, 46, 123, 125, 48, 128, 130, 50, 133, 51, 136, 
138, 53, 141, 54, 144, 146, 56, 149, 151, 58, 154, 59, 157, 159, 61

In the next theorem we will determine the values of $q^{\star}_j$. In the statement and proof of the theorem we will use these notations: For $j\geq 3$, $$M^{\star}(j)=LWF(j),$$ $$m^{\star}(j)=M^{\star}(j)-j+1$$ and $$n^{\star}(j)=\text{max}\{k\mid\{0,1,\dots, k\}\subseteq U^{\star}_j\}.$$

Also, $$L(0)=0, L(1)=3, L(2)=1, L(3)=4, L(4)=2, L(5)=8.$$

In a similar way as we did in \cite{bmmmv}, we will consider also the sets $T^{\star}_1, T^{\star}_2, T^{\star}_3$ to  be the ranges of the sequences $$f^{\star}(n)=UWF(n), g^{\star}(n)=UWF(n)-1\text{ and }h^{\star}(n)=\begin{cases} 5, & \text{if }n=1, \\ 6, & \text{if }n=2, \\ 2LWF(n)+n, & \text{if }n\geq 3,\end{cases}$$respectively. The fisrst values of these sets are:

$$T^{\star}_1=\{2, 4, 8, 10, 12, 15, 18, 21, 23, 26,\dots\}$$

$$T^{\star}_2=\{1, 3, 7, 9, 11, 14, 17, 20, 22, 25,\dots\}$$

$$T^{\star}_3=\{5, 6, 13, 16, 19, 24, 29, 34, 37, 42,\dots\}$$

\begin{theo}\label{propqsj} $q^{\star}_j=L(j)$ for $0\leq j\leq 5$ and, if $j\geq 6$, the following identities hold:
\begin{enumerate}
\item $$W^{\star}_j=\{m^{\star}(j),m^{\star}(j)+1,\dots,M^{\star}(j)\}.$$
\item\label{mjm2} If $j\in T^{\star}_1$, then $n^{\star}(j)=m^{\star}(j)-2$.
\item\label{mjm1} If $j\in T^{\star}_2$, then $n^{\star}(j)=m^{\star}(j)-1$.
\item\label{mj} If $j\in T^{\star}_3$, then $n^{\star}(j)=m^{\star}(j)$.
\item\label{sj} $$S^{\star}_j\cap\mathbb Z^{\geq 0}=\begin{cases} \{0,1,\dots,M^{\star}(j)\}, & \text{if }j\in A^{\star},\\
\{0,1,\dots,M^{\star}(j)\}-\{m^{\star}(j)-1\}, & \text{if }j\in B^{\star}.\end{cases}$$
\item\label{qj} $$q^{\star}_j=\begin{cases} M^{\star}(j)+1, & \text{if }j\in A^{\star},\\
m^{\star}(j)-1, & \text{if }j\in B^{\star}.\end{cases}$$
\end{enumerate}
\end{theo}

We will use the following two lemmas:

\begin{lem}\label{mjmmmo} If $j\geq 5$, then $$M^{\star}(j)-M^{\star}(j-1)=\begin{cases} 2, & \text{if }j-1\in A^{\star},\\
1, & \text{if }j-1\in B^{\star}\end{cases}$$ and $$m^{\star}(j)-m^{\star}(j-1)=\begin{cases} 1, & \text{if }j-1\in A^{\star},\\
0, & \text{if }j-1\in B^{\star}.\end{cases}$$
\end{lem}

\begin{proof} If $j=5$, then $j-1\in B^{\star}, M^{\star}(5)=7, M^{\star}(4)=6, m^{\star}(5)=3, m^{\star}(4)=3$, and the searched equalities hold.

 Let us suppose that $j\geq 6$. We will distinguish four cases:

Case $1$: If $j-1=a(n)$ with $n\not\in F$ then, by part (i) of Lemma 5.3 in \cite{bmmmv}, \begin{equation}\label{ajmajmo} a(j)-a(j-1)=2.\end{equation}
In this case, if $j-1\in F$ then, from (\ref{fifk}) and (\ref{sfifk}), $n\in F$, which contradicts our hypothesis.

Supposse, for contradiction, that $j\in F$, say $j=F_k$. If $k$ is even, then $k-2$ is also even and by (\ref{sfifk}), $b(F_{k-2})=F_k-1$ and $a(n)=F_k-1=b(F_{k-2})$, but this is not possible, because the lower and upper Wythoff sequences partition the set of natural numbers. Thus, $k$ is odd and by (\ref{fifk}), $a(F_{k-1})=F_k-1$. Therefore, $a(n)=F_k-1=a(F_{k-1})$, which contradicts our hypothesis.

Since $j\not\in F$ and $j-1\not\in F$, $\epsilon (j)-\epsilon(j-1)=0$ and, from (\ref{ajmajmo}),  $$M^{\star}(j)-M^{\star}(j-1)=2.$$ Now, the assertion about $m^{\star}(j)-m^{\star}(j-1)$ follows immediately.

Case $2$: If $j-1=a(F_k)$ for some $k$ then, by part (i) of Lemma 5.3 in \cite{bmmmv}, $$a(j)-a(j-1)=2.$$

If $k$ is odd then, by (\ref{fifk}), $a(F_k)=F_{k+1}$ and $\epsilon (j-1)=1$. Since $j\geq 6$, $j\not\in F$, and $\epsilon (j)=0$. Therefore, $\epsilon (j)-\epsilon (j-1)=-1$, and $$M^{\star}(j)-M^{\star}(j-1)=1.$$ If $k$ is even then, by (\ref{fifk}), $a(F_k)=F_{k+1}-1$, is not in $F$, and $\epsilon (j-1)=0$. Also, $j=F_{k+1}$ and $\epsilon (j)=-1$. Therefore, $\epsilon (j)-\epsilon (j-1)=-1$, and $$M^{\star}(j)-M^{\star}(j-1)=1.$$ Now, the assertion for $m^{\star}(j)-m^{\star}(j-1)$ is obvious.

The last two remaining cases, case $3$: $j-1=b(n)$ with $n\not\in F$ and case $4$:  $j-1=b(F_k)$ for some $k$, can be studied in a similar way to the two first ones, and we omit the details.
\end{proof}

\begin{lem}\label{ts1ts2ts3p} The sets $T^{\star}_1, T^{\star}_2, T^{\star}_3$ form a partition of $\mathbb N$.
\end{lem}

\begin{proof} Since, by Theorem \ref{asbsp}, $A^{\star}$ and  $B^{\star}$ is a partition of $\mathbb N$, it is sufficient to prove that $A^{\star}$ is the disjoint union of $T^{\star}_2$ and $T^{\star}_3$. We will prove that, for $n\geq 3$, \begin{equation}\label{lwflwf} LWF(LWF(n))=g^{\star}(n)\end{equation} and, for $n\geq 2$, \begin{equation}\label{lwfuwf} LWF(UWF(n))=h^{\star}(n).\end{equation} We will distinguish three cases to prove (\ref{lwflwf}).

Case $1$: $n=F_k$ with $k$ odd. By using (\ref{fifk}), $$LWF(n)=a(F_k)-1=F_{k+1}-1$$ and, since $F_{k+1}-1\not\in\mathfrak F$, then $$LWF(LWF(n))=a(F_{k+1}-1).$$ On the other side, $$g^{\star}(n)=g^{\star}(F_k)=UWF(F_k)-1=b(F_k)-2.$$ By (\ref{sfifk}), this is $F_{k+2}-2$.

It is well known that lower Wythoff sequence increases in $1$ for positions in $B$, and that $F_{k+1}-1\in B$ when $k$ is odd, and hence $$a(F_{k+1}-1)=a(F_{k+1})-1$$ and, by (\ref{fifk}), this is $F_{k+2}-2$, and we are done by (\ref{sfifk}).

A similar argument, which will be omitted, can be used for the next two cases: Case $2$: $n=F_k$ with $k$ even and Case $3$: $n\not\in\mathfrak F$.

Also, a similar argument, which will be omitted again, can be used for (\ref{lwfuwf}).

Now the assertion that $T^{\star}_2$ and $T^{\star}_3$ partition the set $A^{\star}$ is obvious from Theorem \ref{asbsp}.
\end{proof}

Now we will prove the theorem:

\begin{proof} The first part of the theorem about $q^{\star}_j$ for $0\leq j\leq 5$ can be easily checked from the definition of $q^{\star}_j$.

We will prove the second part by induction on $j$. For $j=6$ the result holds immediately. Let us assume that $j\geq 7$ and that it is true for $j-1$. We have that $$W^{\star}_j=(W^{\star}_{j-1}+1)\cup\{q^{\star}_{j-1}+1\}.$$ By the inductive hypothesis, $$W^{\star}_{j-1}+1=\{m^{\star}(j-1)+1,\dots,M^{\star}(j-1)+1\}.$$ Now we will distinguish two cases: If $j-1\in A^{\star}$ then, by the inductive hypothesis, $$q^{\star}_{j-1}+1=M^{\star}(j-1)+2.$$ By Lemma \ref{mjmmmo}, in this case $m^{\star}(j-1)+1=m^{\star}(j)$ and $M^{\star}(j-1)+2=M^{\star}(j)$, and the assertion about $W^{\star}_j$ is proved. If $j-1\in B^{\star}$ a similar argument is made, using again Lemma \ref{mjmmmo}.

Now we will prove parts (\ref{mjm2}),(\ref{mjm1}),(\ref{mj}). We will distinguish three cases. Lemma \ref{ts1ts2ts3p} guarantees that exactly one of them holds for each $j$.

Case 1: If $j\in T^{\star}_1=B^{\star}$, then $j-1\not\in B^{\star}$ because, for $n\geq 4$, $$UWF(n)-UWF(n-1)=LWF(n)-LWF(n-1)+1$$ and, by Lemma \ref{mjmmmo}, $LWF(n)-LWF(n-1)\geq 1$. Therefore, $j-1\in A^{\star}$. Now, by Lemma \ref{mjmmmo}, $$m^{\star}(j)-m^{\star}(j-1)=1.$$ it holds that $$U^{\star}_j=U^{\star}_{j-1}\cup\{q^{\star}_{j-1}\}$$ and, since $j-1\in A^{\star}$, by the inductive hypothesis, $$q^{\star}_{j-1}=M^{\star}(j-1)+1=m^{\star}(j-1)+j-2.$$ Also, since $j-1\in T^{\star}_2$, by the inductive hypothesis, $n^{\star}(j-1)=m^{\star}(j-1)-1$ and therefore $$q^{\star}_{j-1}\geq n^{\star}(j-1)+2,$$ and $$n^{\star}(j)=n^{\star}(j-1)=m^{\star}(j-1)-1$$ and, by Lemma \ref{mjmmmo}, this is $m^{\star}(j)-2$.

In the other two cases the proof is similar and we will omit it. They are:

Case 2: $j\in T^{\star}_2$, that is divided in two subcases: Subcase 2a: $j-1\in LWF$ and Subcase 2b: $j-1\in UWF$.

Case 3: $j\in T^{\star}_3$. 

Finally, parts (\ref{sj}) and (\ref{qj}) follow directly from the previous ones.
\end{proof}

Although in the previous theorem, in order to match the cases in the partition $T^{\star}_1, T^{\star}_2, T^{\star}_3$, the casuistic $j\leq 5$ and $j\geq 6$ are consider, when describing $q^{\star}_j$ only the cases $0\leq j\leq 2$ need to be considered separately, as we can see in the following obvious corollary of the theorem:

\begin{cor}\label{fqs}
$$q^{\star}_j=\begin{cases} L(j) & \text{if }0\leq j\leq 2,\\
\lfloor\varphi j\rfloor +1, & \text{if }j\in A^{\star}-F,\\
\lfloor\varphi j\rfloor, & \text{if }j\in A^{\star}\cap F,\\
\lfloor (\varphi-1) j\rfloor, & \text{if }j\in B^{\star}-F,\\
\lfloor (\varphi-1) j\rfloor+1, & \text{if }j\in B^{\star}\cap F.\end{cases}$$
\end{cor}

A direct consequence of this corollary is that the elements of the sequence $q^{\star}_j$ distribute asymptotically along two lines centered at the origin with slopes $\varphi$ and $\varphi-1$, depending on if the index is in LWF or UWF, respectively.
In the next figure the plot of the first $100$ terms of the sequence can be seen, where it is possible to note the two mentioned asymptotes:

\begin{figure}[H]
\includegraphics[width= 0.7\textwidth]{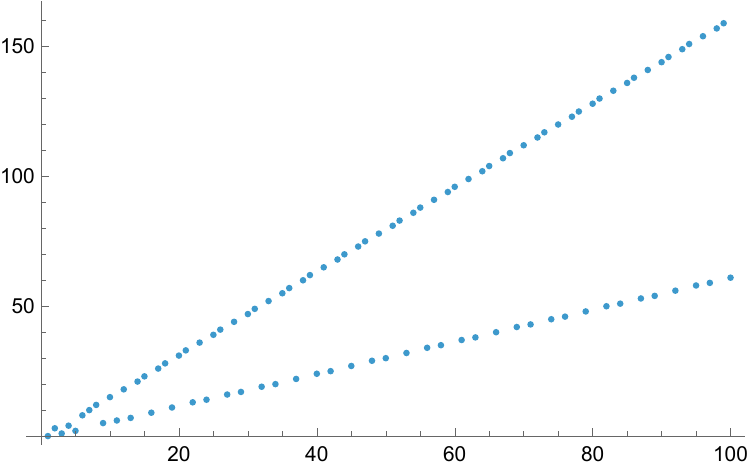}
\caption{Plot of the first $100$ values of the sequence $q^{\star}_j$.}
\end{figure}

\begin{theo} If $j\geq 5$, then $q^{\star}_{q^{\star}_j}=j$.
\end{theo}

\begin{proof} Every $j\geq 5$ can be written in a unique way in exactly one of these two ways: $$j=a(t)+\epsilon (t)\text{ or }j=b(t)+\epsilon (t).$$ It is sufficient to prove that \begin{equation}\label{atpebtpe} q^{\star}_{a(t)+\epsilon (t)}=b(t)+\epsilon (t)\text{ and }q^{\star}_{b(t)+\epsilon (t)}=a(t)+\epsilon (t).\end{equation} We will distinguish three cases:

Case $1$: $t\not\in F$.

Then, $\epsilon (t)=0$, and we need to prove that $q^{\star}_{a(t)}=b(t)$ and $q^{\star}_{b(t)}=a(t)$. This is obvious, because, by (\ref{fifk}) and (\ref{sfifk}), none of $a(t)$ and $b(t)$ is in $F$, and by Corollary \ref{fqs}, the images by $q^{\star}$ of $a(n)$ and $b(n)$ are the same that the respectives images by the sequence $q_n$ with code A002251 in Sloane's database and, as we said before, that sequence swaps the values of the lower and upper Wythoff sequences.

Case 2: $t=F_k$ with $k$ odd.

It holds that $a(t)+\epsilon (t)=a(F_k)-1$ and, by (\ref{fifk}), this is $F_{k+1}-1$, which is in $A^{\star}-F$ and, by Corollary \ref{fqs}, $$q^{\star}_{a(t)+\epsilon (t)}=\lfloor\varphi (F_{k+1}-1)\rfloor+1.$$

On the other hand, $b(t)+\epsilon (t)=b(F_k)-1$ and, by (\ref{sfifk}), this is $F_{k+2}-1$, which is in $B^{\star}-F$ and, by Corollary \ref{fqs}, $$q^{\star}_{b(t)+\epsilon (t)}=\lfloor(\varphi -1) (F_{k+2}-1)\rfloor.$$

Therefore, we have to prove that \begin{equation}\label{A1}\lfloor\varphi (F_{k+1}-1)\rfloor+1=F_{k+2}-1\end{equation} and \begin{equation}\label{A2}\lfloor(\varphi -1) (F_{k+2}-1)\rfloor=F_{k+1}-1.\end{equation} To prove (\ref{A1}), $\varphi(F_{k+1}-1)=\varphi F_{k+1}-\varphi$ and this, by (\ref{pfk}), is $$F_{k+2}+(-1)^{k+2}\varphi^{-k-1}-\varphi=F_{k+2}-\varphi^{-k-1}-\varphi=F_{k+2}-1+(1-\varphi^{-k-1}-\varphi).$$ Since $-1<1-\varphi^{-k-1}-\varphi<0$ for $k\geq 2$, taking floor in $$\varphi(F_{k+1}-1)=F_{k+2}-1+(1-\varphi^{-k-1}-\varphi)$$ we get the desired identity.

Similarly, to prove (\ref{A2}), $(\varphi-1)(F_{k+2}-1)=\varphi F_{k+2}-\varphi-F_{k+2}+1$ and this, again by (\ref{pfk}), is $$F_{k+3}+(-1)^{k+3}\varphi^{-k-2}-\varphi-F_{k+2}+1=F_{k+1}+\varphi^{-k-2}-\varphi+1$$ and since, for $k\geq 1, -1<\varphi^{-k-2}-\varphi+1<0$, taking floor in $$(\varphi-1)(F_{k+2}-1)=F_{k+1}+\varphi^{-k-2}-\varphi+1$$ we obtain the desired identity.

Finally, Case 3: $t=F_k$ with $k$ even, can be treated in a similar way to Case $2$, and we will omit it.
\end{proof}

\begin{cor} $q^{\star}_j$ is a permutation of the set $\mathbb Z^{\geq 0}$ of non-negative numbers.
\end{cor}

\begin{proof} It fixes $0$, is an involution on $\mathbb Z^{\geq 5}$ and, restricted to $\{1,2,3,4\}$ is the cycle $(1,3,4,2)$.
\end{proof}

Therefore, it is a permutation of order $4$ that is an involution when restricted to $\mathbb Z^{\geq 5}$. We will call to such a permutation an almost-involution.

In the next theorem and its corollary, we will prove that, as happened with the sequence $q_j$, also when we consider the sequence $q^{\star}_j-j$ every integer number appears exactly once:

\begin{theo} For every $i\geq 3$ $$q^{\star}_{LWF(i)}-LWF(i)=i$$ and $$q^{\star}_{UWF(i)}-UWF(i)=-i.$$
\end{theo}

\begin{proof} By (\ref{atpebtpe}), $$q^{\star}_{a(t)+\epsilon (t)}=b(t)+\epsilon (t),$$ and $$b(t)+\epsilon(t)-(a(t)+\epsilon (t))=b(t)-a(t)=t,$$ and this proves the first identity. The second one is proved in a similar way, and we will omit it.
\end{proof}

\begin{cor} The sequence $q^{\star}_j-j$ contains every integer once.
\end{cor}

\begin{proof} It holds that $$q^{\star}_0-0=0, q^{\star}_1-1=2, q^{\star}_2-2=-1, q^{\star}_3-3=1, q^{\star}_4-4=-2$$ and, for $j\geq 5$, by the previous theorem, every integer with absolute value at least $3$ appears once.
\end{proof}

In a similar way as happened with $q_n$ and also with Shapovalov's and Venkatachala's sequence, an interpretation can be given for $q^{\star}_j$ as a greedy perturbed sequence in which the arithmetic mean of its terms are always integer numbers:

\begin{theo} For $j\geq 3$, $q^{\star}_j$ is the least non-negative integer different from $q^{\star}_0,\dots,q^{\star}_{j-1}$ for which the arithmetic mean $\frac{q^{\star}_0+\dots+q^{\star}_j}{j+1}$ takes an integer value.
\end{theo}

\begin{proof} It can be easily checked that the theorem is true for $3\leq j\leq 5$. So, let us suppose that $j\geq 6$. By part (1) of Theorem \ref{propqsj}, $$\{q^{\star}_0+j,q^{\star}_1+j-1,\dots,q^{\star}_{j-1}+1\}=\{m^{\star}(j),m^{\star}(j)+1,\dots,m^{\star}(j)+j-1\}.$$ Since both sets are equal, also is the sum of their elements, that is, $$q^{\star}_0+q^{\star}_1+\dots+q^{\star}_{j-1}+\frac{j(j+1)}{2}=j m^{\star}(j)+\frac{j(j-1)}{2},$$ and therefore $$q^{\star}_0+q^{\star}_1+\dots+q^{\star}_{j-1}=j(m^{\star}(j)-1).$$ If we denote now by $x_j$ the minimum value distinct from $q^{\star}_0,\dots,q^{\star}_{j-1}$ for which $$\frac{q^{\star}_0+\dots+q^{\star}_{j-1}+x_j}{j+1}\in\mathbb Z$$ then, since $j(m^{\star}(j)-1)\equiv 1-m^{\star}(j)\pmod{j+1}$, we obtain that the condition that the mean is an integer is equivalent to \begin{equation}\label{cmint}x_j\equiv m^{\star}(j)-1\pmod{j+1}.\end{equation} Now we will distinguish two cases:

Case $1$: If $j\in B^{\star}=T^{\star}_1$ then, by part (2) of Theorem \ref{propqsj}, $n^{\star}(j)=m^{\star}(j)-2$, but then $\{0,1,\dots,m^{\star}(j)-2\}\subseteq U^{\star}_j$ and $m^{\star}(j)-1\not\in U^{\star}_j$, and the least $x_j$ satisfying (\ref{cmint}) and different from $q^{\star}_0,\dots,q^{\star}_{j-1}$ is $m^{\star}(j)-1$.

Also, by part (6) of Theorem \ref{propqsj}, $q^{\star}_j=m^{\star}(j)-1$.

Case 2: If $j\in A^{\star}=T^{\star}_2\cup T^{\star}_3$ then, by parts (3) and (4) of Theorem \ref{propqsj}, $n^{\star}(j)\geq m^{\star}(j)-1$, and therefore $m^{\star}(j)-1\in U^{\star}_j$, and the least $x_j$ different from $q^{\star}_0,\dots,q^{\star}_{j-1}$ satisfying (\ref{cmint}) is $m^{\star}(j)+j$.

Again, by part (6) of Theorem \ref{propqsj}, $q^{\star}_j=M^{\star}(j)+1=m^{\star}(j)+j$.
\end{proof}

\section{Future research}

Using Algorithm \ref{alg} it is very hard to obtain closed formulas for the values in the infinite matrix obtained. The main difficulty is the complex structure of the set $S\cap\mathbb Z^{\geq 0}$, because of the abundance of holes in that set, which makes difficult to obtain explicitely the values of the mex. Nonetheless, when doing compatible perturbations in the initial values, for the third row the set seems to have a nice structure in which, for large $j$, it is eventually either a full interval of integers or has a single hole. We think that it would be interesting doing perturbations in initial segments of this third row, preserving in the perturbed values the property that there are no repetitions in the $q_j$ values and in the $q_j-j$ values.

\end{document}